\documentclass[a4paper, 11pt, headings = small, abstract,numbers=endperiod]{scrartcl}

\usepackage{amscd,amsmath,amsfonts,amsthm,array,bbm,hhline,mathrsfs, enumerate,dsfont,changes, booktabs,fancybox,calc,textcomp,xcolor,graphicx,bbm,xspace,nicefrac,stmaryrd,url,arcs,listings,nameref}

\usepackage[title]{appendix}
\usepackage[semibold]{libertine}
\usepackage[upint]{libertinust1math}
\usepackage[scr=boondoxupr, scrscaled = 1.0]{mathalpha}

\usepackage[colorlinks=true, allcolors=myteal]{hyperref}
\definecolor{WIMgreen}{RGB}{60 134 132}
\definecolor{red_pers}{RGB}{204 37 41}
\definecolor{UMblue}{RGB}{4 47 86}
\definecolor{myteal}{RGB}{0 123 137}
\definecolor{dartmouthgreen}{rgb}{0.05, 0.5, 0.06}\definecolor{cobalt}{rgb}{0.0, 0.28, 0.67}\definecolor{coolblack}{rgb}{0.0, 0.18, 0.39}
\definecolor{glaucous}{rgb}{0.38, 0.51, 0.71}\definecolor{hooker\'sgreen}{rgb}{0.0, 0.44, 0.0}\definecolor{lemonchiffon}{rgb}{1.0, 0.98, 0.8}\definecolor{oucrimsonred}{rgb}{0.6, 0.0, 0.0}\definecolor{radicalred}{rgb}{1.0, 0.21, 0.37}\definecolor{raspberry}{rgb}{0.89, 0.04, 0.36}\definecolor{royalazure}{rgb}{0.0, 0.22, 0.66}
\definecolor{dex}{RGB}{138 18 34}
\definecolor{darkgreen}{RGB}{0 69 0}
\definecolor{darkblue}{RGB}{0 0 99}

\usepackage[utf8]{inputenc}
\usepackage[ngerman,english]{babel}								
\usepackage[numbers, square]{natbib}

\usepackage{mathtools}
\usepackage{graphicx}

\newcommand{\R}{\mathbb{R}}

\newcommand{\N}{\mathbb{N}}

\newcommand{\Ex}{\mathbb{E}}
\newcommand{\PR}{\mathbb{P}}

\newcommand{\ind}[1]{\mathds{ 1 }_{\{{#1}\}}}
\newcommand{\inda}[1]{\mathds{ 1 }_{{#1}}}
\newcommand*{\f}{\mathcal{F}}
\newcommand{\eps}{\varepsilon}
\DeclareMathOperator{\Exp}{Exp}

\DeclareMathOperator{\ccc}{C}
\DeclareMathOperator{\B}{B}

\DeclareMathOperator{\rm}{RM}
\DeclareMathOperator{\id}{id}

\makeatletter
\newcommand*\bigcdot{\mathpalette\bigcdot@{.6}}
\newcommand*\bigcdot@[2]{\mathbin{\vcenter{\hbox{\scalebox{#2}{$\,\m@th#1\bullet\,$}}}}}
\makeatother

\def\XXint#1#2#3{{\setbox0=\hbox{$#1{#2#3}{\int}$ }
\vcenter{\hbox{$#2#3$ }}\kern-.6\wd0}}

\theoremstyle{definition}

\newtheorem*{remark*}{Remark}
\newtheorem{theorem}{Theorem}

\newtheorem{lemma}[theorem]{Lemma}

\title{\fontsize{16}{19} \selectfont 		On differentiability of reward functionals corresponding to Markovian randomized stopping times
}
\author{Boy Schultz\thanks{Kiel University, Mathematical Department\\email: \href{mailto:schultz@math.uni-kiel.de}{schultz@math.uni-kiel.de}}}
\date{November 18, 2024}

\begin{document}

\maketitle
\begin{abstract}
We conduct an investigation of the differentiability and continuity of reward functionals associated to Markovian randomized stopping times. Our focus is mostly on the differentiability, which is a crucial ingredient for a common approach to derive analytic expressions for the reward function.
\end{abstract}
\noindent
\small{\textit{\href{https://mathscinet.ams.org/mathscinet/msc/msc2020.html}{2020 MSC:}} Primary 60G40; secondary 60J35, 60J55\\
\textit{key words:} reward functional, Markovian randomized stopping times, linear diffusion, regularity, differentiability}
\normalsize

\section{Introduction}

A common problem in stochastic calculus is to derive analytic expressions for reward functionals 
\begin{align*}
    J_\tau(x):=\Ex_x[e^{-r\tau} g(X_\tau)],
\end{align*} with a given stopping time $\tau$. In the literature this problem has been discussed extensively for a variety processes $X$, functions $g$ and stopping times $\tau$, see e.g.\ \cite{dynkin_markov2,kyprianou2006introductory,oksendal, peskir}. Here, let us first consider the case where $X=(X_t)_{t\in[0,\infty)}$ is a linear diffusion with values in an interval $I$, $r\ge0$ is a discount factor and $g$ a Hölder continuous payoff function. The solution to our problem is most well know if $\tau$ is a first exit time of some open interval $(a,b)\subset I$, see e.g.\ \cite[Section 9]{oksendal}. 

Due to their connection to subgame perfect Nash equilibria, so called \textit{Markovian randomized stopping times\footnote{The origin of the name is explained in \cite[Section 3.2]{decamps2022mixed}}} are of particular interest in the context of games of stopping and were recently considered in \cite{BodnariuChristensenLindensjoe2022,christensen2023time, christensen2024existencemarkovianrandomizedequilibria, decamps2022mixed, ekstrom2017dynkin, riedel2017subgame}. This class of stopping times will be formally introduced in Section \ref{sec:model} and contains, among others, all stopping times of the form
\begin{align}
    \tau = \inf\bigg\{ t\ge 0: \int_0^t\psi(X_s)ds \ge E \,\text{ or }\, X_t\not \in (a,b)\bigg\} \label{int:2}
\end{align} with $a<b\in I$, a Hölder continuous function $\psi: I \to [0,\infty)$ and an exponentially distributed random variable $E$. Here, $\psi$ plays the role of an infinitesimal stopping rate. Note that first exit times of open intervals are a contained as a special case. For $\tau$ given by \eqref{int:2} the corresponding reward function $J_\tau$ is typically determined as the unique solution to the ordinary differential equation
\begin{align}
     (A-r-\psi) J_\tau = -g\psi \label{int:3}
\end{align} on $(a,b)$ with boundary conditions
\begin{align}
     \quad J_\tau(a)=g(a),\quad J_\tau(b)=g(b)\label{int:4}
\end{align} where $A$ denotes the differential generator of the diffusion $X$, see \cite[Theorem 13.16]{dynkin_markov2}. Provided that the coefficients of the differential operator $A$ are sufficiently nice, this allows to to determine an analytic expression for $J_\tau$. 

However, if $\psi$ is merely piecewise Hölder continuous $J_\tau$ must no longer be $\ccc^2$ in the  jumps $\psi$. Thus \eqref{int:3} cannot hold as a second order ODE on $(a,b)$ in the classical sense. This situation is encountered in \cite{christensen2023time}, where equilibrium stopping rates $\psi$ are piecewise Hölder continuous by the nature of the problem. Generally, discontinuous stopping rates $\psi$ could be considered in any kind of stopping game in which Markovian randomized stopping times are equilibrium candidates, such as \cite{BodnariuChristensenLindensjoe2022} or \cite{christensen2020time}.

Now let $\psi$ only satisfy Hölder conditions on $(x_0,x_1),(x_1,x_2),(x_2,x_3),...,(x_{n-1},x_n)$, $x_0:=a,x_n:=b$. General theory still provides that on each interval $(x_{i-1},x_i)$, $i=1,...,n$ the function $J_\tau$ is twice continuously differentiable and satisfies \eqref{int:3}, see Theorem \ref{thm:main}, \eqref{diff2}. To salvage the approach and recover an analytic expression for $J_\tau$ from \eqref{int:3} posed on $(x_{i-1},x_{i})$, $i=1,...,n$ we need appropriate boundary conditions on each interval. Since \eqref{int:3} is a second order differential equation, as a rule of thumb, we need two boundary conditions for each interval $(x_{i-1},x_{i})$, $i=1,...,n$. Taking into account \eqref{int:4} we still need $2n - 2$ conditions. It is comparatively simple to show that $J_\tau$ is continuous on $(a,b)$ which yields the first $n-1$ conditions, see Theorem \ref{thm:main}, \eqref{cont}. The challenging part is to show that $J_\tau$ is even differentiable on $(a,b)$ which gives the remaining $n-1$ conditions, see Theorem \ref{thm:main}, \eqref{diff}.

In fact it is surprisingly simple to put that idea in more precise terms. For that let $\tau^{(a,b)}:=\inf\{ t \ge 0: X_t \not \in (a,b)\}$ and assume $\Ex_x[\tau^{(a,b)}]<\infty$. Now, if $\tilde J$ is $\ccc^0$ on $[a,b]$, $\ccc^1$ on $(a,b)$, $\ccc^2$ on $\bigcup_{i\in\{1,...,n\}} (x_{i-1},x_i)$ and satisfies \eqref{int:3} on $(x_{i-1},x_i)$, $i=1,...,n$ as well as \eqref{int:4}, then $\tilde J = J_\tau$. This follows from the subsequent computation, which will by justified step by step below. We have
\begin{align*}
    \tilde J(x) 
    =& \Ex_x[ e^{-r\tau^{(a,b)} -\int_0^{\tau^{(a,b)}} \psi(X_t)dt} J(X_{\tau^{(a,b)}})]\\
    &-\Ex_x\left[ \int_0^{\tau^{(a,b)}} \ind{X_t\not\in \{x_1,...,x_{n-1}\}}e^{-rt -\int_0^t \psi(X_s)ds} (A-r-\psi) \tilde J(X_t)  dt \right] \\
    =& \Ex_x[ e^{-r{\tau^{(a,b)}} -\int_0^{\tau^{(a,b)}} \psi(X_t)dt} g(X_{\tau^{(a,b)}})] +\Ex_x\left[ \int_0^{\tau^{(a,b)}} e^{-rt -\int_0^t \psi(X_s)ds} \psi(X_t)g(X_t)  dt \right] \\
    =& J_\tau(x).
\end{align*}
In the first step we use a version of Dynkin's formula based on the generalized Itô formula \cite[Chapter IV, Theorem 71]{protter2005stochastic}. This formula requires that $\tilde J$ is $\ccc^1$ with absolutely continuous derivative. Here, the absolute continuity of $\tilde J$ is implied by the piecewise $\ccc^2$ assumption. Note that Dynkin's formula also involves the assumption $\Ex_x[\tau^{(a,b)}]<\infty$. Besides that we use that $\Ex_x[\int_0^{t}\ind{X_s\in \{x_1,...,x_{n-1}\}}ds]=0$ for all $t\ge 0$. In the second step we apply \eqref{int:3} and \eqref{int:4}. The last equality follows from \eqref{eq:zerl2} which will be shown later on.

The main contribution of the present paper is a proof of the $\ccc^1$ property of $J_\tau$ with $\tau$ given by \eqref{int:2}, but with $\psi$ being merely piecewise Hölder continuous, which is a crucial part of the previous argument. Additionally, we show continuity of $J_\tau$ for general Markovian randomized stopping times $\tau$ under very mild conditions. Both results are stated in Theorem \ref{thm:main} next to a known result which provides sufficient conditions for $J_\tau$ to be $\ccc^2$.

\section{General framework}\label{sec:model}
We consider a regular linear Itô diffusion $X = (X_t)_{t \in [0,\infty)}$ taking values in an interval $I \subset \R$ and defined on a filtered probability space $(\Omega, \f, (\f_t)_{t \in [0,\infty)},\PR)$ satisfying the usual hypotheses. Generally we assume that  the behavior of $X$ in the interior $I^\circ$ of $I$ is governed by the stochastic differential equation
\begin{align}
    d X_t = \mu(X_t) dt + \sigma(X_t) dW_t , \quad X_0=x\in I^\circ\label{eq:dynamics}
\end{align} with an $(\f_t)_{t \in [0,\infty)}$-adapted, real valued standard Brownian motion $W = (W_t)_{t \in [0,\infty)}$ and Lip\-schitz continuous coefficients $\mu: I \to \R$, $\sigma: I \to (0,\infty)$. A jointly continuous version of the local time process of $X$ at $y\in I$ will be denoted by $(l^y_t)_{t \in [0,\infty)}$. The existence of such a version follows e.g. from \cite[Theorem (44.2)]{RogersWilliamsVol2}. Let $E\sim \Exp(1)$ be a random variable on $(\Omega, \f, \PR)$ which is independent from $X$. We set $\f^X_\infty := \sigma(X_t:t\ge0)$ and denote the canonical shift operator associated to $X$ by $\theta$. If $Y$ is any random variable on $(\Omega, \f, (\f_t)_{t \in [0,\infty)},\PR)$ with values in some metric space, we denote its distribution by $\PR^Y$. As usual, we write $\PR_x$ for the conditional distribution of $\PR$ given $X_0=x$ and $\Ex_x$ for the corresponding expectation operator. We write $\Ex_{\bigcdot}[...]$ for the function $ I\ni x\mapsto \Ex_x[...]$.

For open or closed $D\subset I$ we denote the Borel $\sigma$-algebra on $D$ by $\mathcal{B}(D)$, the space of Radon measures on $D$ by $\rm(D)$ and set $\B(D):=\{f:D\to \R: f$ is bounded and measurable$\}$, $\ccc^0(D):=\{f:D\to \R: f$ is continuous$\}$, $\ccc^m(D):=\{f:D\to \R: f$ is $m$ times continuously differentiable$\}$. If $f:D\to [0,\infty)$ is Lebesgue integrable, we denote the measure that maps $\Gamma\in \mathcal{B}(D)$ to $\int_\Gamma f dx$ by $f dx$. For $D'\subset D \subset I$ and functions $f:D\to \R$ or measures $\lambda\in \rm (D)$ we denote the restrictions of $f$ and $\lambda$ to $D'$ by $f\vert_{D'}$ and $\lambda\vert_{D'}$ respectively. We call $f:D\to \R$ piecewise Hölder continuous, if there are $\inf D = x_0 <x_1< ... < x_n =\sup D$, $n\in \N$ such that $f\vert_{D\cap (x_{i-1},x_i)}$ is Hölder continuous (possibly with different exponents) for each $i\in \{1,...,n\}$. 

We denote the first exit time from open $D\subset I$ by
\begin{align*}
    \tau^D &:= \inf \{ t \ge 0 : X_t \not\in D\}
\end{align*} For $D \subset I$ open (in $I$) and $\lambda \in \rm(D)$. We set the additive functional $A^{D,\lambda} = (A^{D,\lambda}_t)_{t \in [0,\infty)}$ generated by $X,D$ and $\lambda$ to be given by
\begin{align}
    A^{D,\lambda}_t(\omega):= \int_D l^y_t(\omega) \lambda(dy) + \infty \ind{\tau^D(\omega) \le t},\quad t \ge 0.\label{func}
\end{align}  We define the \textit{Markovian randomized (stopping) time} generated by $D$, $\lambda$ and $E$ as
\begin{align}
    \tau^{D,\lambda} := \inf\{ t \ge 0: A_t^{D,\lambda} \ge E\}.\label{eq:taudlambda}
\end{align}The space of all Markovian randomized stopping times based on the random variable $E$ is denoted by
\begin{align*}
    \mathcal{Z}:= \mathcal{Z}(E):=\{ \tau^{D,\lambda}: D\subset I \text{ open (in } I),  \lambda\in \rm(D)\}.
\end{align*}  

Markovian randomized stopping times $\tau\in \mathcal Z$ feature a Markov property of the following type: With a natural extension of the shift operator $\theta$ it holds that
\begin{align*}
    \ind{\tau \ge \sigma} \tau = \ind{\tau\ge\sigma} \theta_\sigma \circ \tau + \sigma
\end{align*} in distribution for all $\f^X_\infty$-measurable $\sigma$, see \cite{christensen2023time}. 
More detailed discussions can be found in \cite[Section 3]{decamps2022mixed}.

We fix a measurable function $g:I\to \R$ and a constant $r\ge 0$. Based on that define the \textit{reward functional} corresponding to $\tau\in \mathcal{Z}$ via
\begin{align*}
    J_{\tau}(x):=J(x,\tau):= \Ex_x[e^{-r\tau}g(X_{\tau})],\quad x\in I
\end{align*} whenever the right hand side exists possibly taking value $\pm \infty$. Note that 
\begin{align*}
    J_{\tau^{D,\lambda}}(x) = g(x)
\end{align*} for all $\tau^{D,\lambda}\in \mathcal{Z}$ and all $x\in I\setminus D$, so the behavior of $J_{\tau^{D,\lambda}}$ on $I\setminus D$ is predetermined by $g$. Thus, in the following we are only concerned with differentiability on $D$ and with continuity on the closure $\overline{D}$ of $D$ (in $I$).

\section{Main result}
\begin{theorem}\label{thm:main}
Let $\tau=\tau^{D,\lambda}\in \mathcal{Z}$, $A:=A^{D,\lambda}$ and $\inf I \le a<b\le \sup I$ such that $(a,b)\subset D$ and $\Ex_x[\tau^{(a,b)}]<\infty$ for all $x\in (a,b)$. 
\begin{enumerate}[(i)]
    \item \label{cont} If $g$ is bounded on $(a,b)$, $\lambda((a,b))<\infty$ and $J_\tau(y)\in \R$ for all $y\in I\cap \{a,b\}$ then $J_\tau\vert_{[a,b]}\in \ccc^0([a,b])$.
    
    \item \label{diff2} Let $\psi:(a,b)\to [0,\infty)$ be Hölder continuous. Suppose that $\lambda\vert_{(a,b)}= \frac{\psi}{\sigma}dx$, that $g$ is Hölder continuous on $(a,b)$ and that $J_\tau(y)\in \R$ for all $y\in I\cap \{a,b\}$. Then $J_\tau\vert_{(a,b)}\in \ccc^2((a,b))$.

    \item \label{diff} Let $\psi:(a,b)\to [0,\infty)$ be piecewise Hölder continuous. Suppose that $\lambda\vert_{(a,b)}= \frac{\psi}{\sigma}dx$, that $g$ is Hölder continuous on $(a,b)$ and that $J_\tau(y)\in \R$ for all $y\in I\cap \{a,b\}$. Then $J_\tau\vert_{(a,b)}\in \ccc^1((a,b))$.
\end{enumerate}
\end{theorem}

\begin{remark*}
    
    In fact our main contribution is \eqref{diff}. The other two statements are mostly for completeness, with \eqref{cont} being less involved on our part and \eqref{diff2} a direct consequence of a result from \cite{dynkin_markov2}. Continuity of the functional $J$ as well as related functionals including those treated in Lemma \ref{lemma:cont} has been studied extensively, e.g.\ by \cite{dynkin_markov, dynkin_markov2, ito1974diffusion}. Assembling suitable results from the various sources already leads us most of the way towards \eqref{cont}.
    
    The proof of \eqref{diff} consists of two major parts. The first one is to show that $J_\tau$ is in the domain of the characteristic operator of $X$ (defined according to \cite{dynkin_markov}). The second part is to put this together with \eqref{cont} and \eqref{diff2} to prove the claim. 
    The first part follows a line of arguments from \cite{dynkin_markov, dynkin_markov2}. Lemma \ref{lemma:thm9.7}, Lemma \ref{lemma:thm13.11} and Lemma \ref{lemma:thm13.12} are adaptations of theorems from \cite{dynkin_markov, dynkin_markov2} that accommodate for our deviating general assumptions, in particular that we allow discontinuity of $\psi$. 
\end{remark*}

In the following we use the terms infinitesimal operator, weak infinitesimal operator, characteristic operator, resolvent, potential, standard process, continuous homogeneous multiplicative functional, (standardized) $\alpha$-subprocess, canonical diffusion and weak convergence according to the definitions in \cite{dynkin_markov}. We merely note that for open $D\subset I$ this means weak convergence of a sequence $(f_n)_{n\in \N}$ in $\B(D)$ to some $f\in \B(D)$ is defined via the property $\int f_n d\mu \to \int fd\mu$ for all finite measures $\mu$ on $D$.
We denote the domain of a (weak) infinitesimal operator $A$ by $\mathcal{D}_A$. Similarly, for the characteristic operator $\mathfrak A$ of a Markov process with values in $J \subset I$ we denote the set of all measurable functions such that $\mathfrak A f(x)$ exists for a fixes $x\in J$ by $\mathcal{D}_{\mathfrak A}(x)$. Moreover, the domain of $\mathfrak A$, i.e.\ the set of all measurable functions such that $\mathfrak A f(y)$ exists for all $y\in J$, is denoted by $\mathcal{D}_{\mathfrak A}$. We will also use the notation for Markov processes from \cite[Subsection 3.1]{dynkin_markov}. There, a Markov process is denoted as a quadruple  $Y=(Y_t, \zeta, \mathcal{M}_t,Q_x)=((Y_t)_{t \in [0,\infty)}, \zeta, (\mathcal{M}_t)_{t \in [0,\infty)},(Q_x)_{x\in \mathcal I})$ consisting of paths $(Y_t)_{t \in [0,\infty)}$, lifetime $\zeta$, filtration at lifetime $(\mathcal{M}_t)_{t \in [0,\infty)}$, i.e. $\mathcal{M}_t$ is a $\sigma$-algebra on $\{\omega \in \Omega: \zeta(\omega)>t\}$ with $A\in \mathcal{M}_s$ implying $A\cap \{\zeta >s'\}\in \mathcal{M}_{s'}$ for all $s\le s'$ and transition probabilities $Q_x$, $x\in \mathcal I$ with $\mathcal I$ denoting the state space of $Y$. In particular, the process $X$ from Section \ref{sec:model} is a standard process which reads $X=(X_t,\infty, \f_t,\PR_x)$.

The next lemma is concerned with the continuity of two expectations which are closely related to $J_\tau$ from Theorem \ref{thm:main}. It will not only lead us most of the way towards the proof of Theorem \ref{thm:main} \eqref{cont} but also be used to show that Lemma \ref{lemma:thm9.7} can be applied in the proofs of Lemma \ref{lemma:thm13.11} and Lemma \ref{lemma:thm13.12}. 

\begin{lemma}\label{lemma:cont}
Let $\tau^{D,\lambda}\in \mathcal{Z}$, $A:=A^{D,\lambda}$ the corresponding additive functional given by \eqref{func}, $\tilde \lambda \in \rm(D)$, $\tilde A:= A^{D,\tilde \lambda}$ the functional corresponding to $D,\lambda$ given by \eqref{func}, $J\subset D$ an open interval, $\tau:=\tau^J$ and $h:I\to \R$ a bounded measurable function. Additionally, we set $\tilde A_{t-}:= \lim_{s\nearrow t} \tilde A_s$ with convention $\tilde A_{0-}:=0$.
\begin{enumerate}[(i)]
    \item \label{lemma:cont1} If $\tau<\infty$ $\PR_x$-a.s.\ for all $x\in J$, then the function $f:I \to \R$,
\begin{align*}
    x\mapsto \Ex_x[e^{-A_\tau}h(X_\tau)]
\end{align*} is continuous on the closure $\overline{J}$ of $J$ (in $I$).
\item \label{lemma:cont2} If $\Ex_x[\tau]<\infty$ for all $x\in J$, $\mu\vert_J$, $\sigma\vert_J$ are bounded and $\tilde \lambda (J)<\infty$, then the function $F: I \to \R$,
\begin{align*}
    x\mapsto \Ex_x\left[ \int_0^\tau e^{-A_t} h(X_t) d\tilde A_t \right]
\end{align*} is continuous on $I$. We use the convention $\int_0^\tau:=\int_{[0,\tau)}$.
\end{enumerate}
\end{lemma}

\begin{proof} \eqref{lemma:cont1} Since $h$ is bounded, $f$ is clearly well defined. Set $J:= (a,b)$ for $a<b\in [-\infty,\infty]$. Note that $a,b$ are not necessarily in $I$.

For $y\in \R$ let ${\eta_y}:=\inf\{t\ge 0 : X_t = y \}$, with convention $\inf \varnothing := \infty$. The strong Markov property yields
\begin{align*}
    f(x) 
    &= \Ex_x[\ind{{\eta_y} \le \tau }e^{-(\theta_{\eta_y} \circ A_\tau + A_{\eta_y})}h(X_{\theta_{\eta_y} \circ A_\tau + A_{\eta_y}})] +  \Ex_x[\ind{{\eta_y} > \tau }e^{-A_\tau}h(X_\tau)] \\
    &= \Ex_x[ \ind{{\eta_y} \le \tau }e^{- A_{\eta_y} }\Ex_x[e^{-\theta_{\eta_y} \circ A_\tau} h(X_{\theta_{\eta_y} \circ A_\tau + A_{\eta_y}})  \vert \f_{{\eta_y}} ]]+  \Ex_x[\ind{{\eta_y} > \tau }e^{-A_\tau}h(X_\tau)] \\
    &=  \Ex_x[ \ind{{\eta_y} \le \tau }e^{- A_{\eta_y} }\Ex_{X_{\eta_y}}[e^{-A_\tau}h(X_\tau)] +\Ex_x[\ind{{\eta_y} > \tau }e^{-A_\tau}h(X_\tau)]\\
    &=  \Ex_x[ \ind{{\eta_y}\le  \tau }e^{- A_{\eta_y} }] \Ex_{y}[e^{-A_\tau}h(X_\tau)] +\Ex_x[\ind{{\eta_y} > \tau }e^{-A_\tau}h(X_\tau)]\\
    &=  \Ex_x[ \ind{{\eta_y} \le \tau }e^{- A_{\eta_y} }] f(y) +\Ex_x[\ind{{\eta_y} > \tau }e^{-A_\tau}h(X_\tau)]
\end{align*} for all $x,y\in I$. Thus
\begin{align}
    \vert f(x)- f(y)\vert &= \vert \Ex_x[ \ind{{\eta_y} \le \tau }e^{- A_{\eta_y} }] f(y) +\Ex_x[\ind{{\eta_y} > \tau }e^{-A_\tau}h(X_\tau)] - f(y)\vert \notag \\
    &\le \vert f(y) \vert (1- \Ex_x[\ind{\eta_y \le \tau} e^{-A_{\eta_y}}]) + \Big( \sup_{z\in I} \vert h(z)\vert \Big) \Ex_x[\ind{\eta_y > \tau}]  \label{eq:est1}
\end{align} for all $x,y\in I$. By symmetry it suffices to show that both summands on the right hand side of \eqref{eq:est1} go to 0, whenever $J \ni x\to y\in  \overline{J}$.

We start with the first summand. Since $X$ is a regular diffusion, $\lim_{J \ni x\to y}\PR_x(\eta_y< \eps)=1 $ for all $\eps>0$, cf.\ \cite[Section 3.3, 10c)]{ito1974diffusion}. By continuity of paths $\PR_x(\tau>0)=1$ for all $x\in J$. By construction $t\mapsto A_t$ is continuous on $[0,\tau^D)$ with $A_0=0$. Putting that together we find that $\Ex_x[\ind{\eta_y \le \tau} e^{-A_{\eta_y}}]\stackrel{J\ni x\to y}{\longrightarrow} 0$.

We treat the second summand next. If $y\in \{a,b\}$, then $\Ex_y[\ind{\eta_y>\tau}]=0$ and we are done. For $y\in (a,b)$ we have $\Ex_x[\ind{\eta_y> \tau}]=\PR_x(\eta_y >\eta_b)$ if $y\le x\le b$ and $\Ex_x[\ind{\eta_y> \tau}]=\PR_x(\eta_y >\eta_a)$ if $a\le x\le y$. By \cite[Section 4.4]{ito1974diffusion} the functions $[y,b]\ni x\mapsto \PR_x(\eta_y >\eta_b)$ and $[a,y]\ni x \mapsto \PR_x(\eta_y >\eta_a)$ are continuous with $\PR_x(\eta_x >\eta_b)=0$ and $\PR_x(\eta_x >\eta_a)=0$, which proves the claim.

\eqref{lemma:cont2} Let $\psi: J \to \R$, $x\mapsto \tilde \lambda ((a,x])$ and $\Psi:J \to \R$, $x\mapsto \int_0^x\psi(y) dy$. Note that $\tilde \lambda(J)<\infty$ implies that $\psi$ is bounded and $\Psi$ is continuous and bounded. For $x\in J$ and $t <\tau^D $ the Itô-Meyer formula, cf. \cite[Chapter 3, Theorem 70]{protter2005stochastic}, provides
\begin{align}
   \tilde A_t 
   &= \int_J l^y_t \tilde \lambda(dy) 
   = 2\bigg(\Psi(X_t) - \Psi(x) - \int_0^t \psi(X_s) dX_s \bigg)\notag\\
   &= 2\Psi(X_t) -2 \Psi(x) - \int_0^t 2\psi(X_s) \mu(X_s) ds - \int_0^t 2 \psi(X_s) \sigma(X_s) dW_s. \notag
\end{align} $\PR_x$-a.s.
By continuity of the paths of $A$ on $[0,\tau^D)$ in particular 
\begin{align}
    \tilde A_{t-}=  2\Psi(X_{t}) -2 \Psi(x) - \int_0^{t} 2\psi(X_s) \mu(X_s) ds - \int_0^{t} 2 \psi(X_s) \sigma(X_s) dW_s \label{eq:itomeyer}
\end{align} $\PR_x$-a.s.\ for all $x\in J$ and all $t\le \tau^D$. We find that
\begin{align*}
    &F(x) 
    \le \Big( \sup_{z\in I} \vert h(z) \vert \Big) \Ex_x\bigg[\int_0^\tau d \tilde A_\tau\bigg] 
    = \Big( \sup_{z\in I} \vert h(z) \vert \Big) \Ex_x[\tilde A_{\tau-}]\\
    =& \Big( \sup_{z\in I} \vert h(z) \vert \Big) \left(\Ex_x[2\Psi(X_{\tau})]-2 \Psi(x) -2\Ex_x\left[ \int_0^{\tau} 2\psi(X_s) \mu(X_s) ds \right]-\Ex_x\left[ \int_0^{\tau} 2 \psi(X_s) \sigma(X_s) dW_s\right]  \right) 
\end{align*} for all $x\in I$. Since $h,\psi, \Psi, \mu\vert_J$ and $ \sigma\vert_J$ are bounded and $\Ex_x[\tau]<\infty$ the function $F$ is well defined. 

For all $x,y\in I$ the strong Markov property yields
\begin{align}
    F(x)
    =& \Ex_x\left[\ind{\eta_y \le \tau } \int_0^{\eta_y} e^{-A_t}h(X_t) d\tilde A_t \right] + \Ex_x\left[ \ind{\eta_y \le \tau }\Ex_x \left[ \int_{\theta_{\eta_y}\circ \eta_y}^{\theta_{\eta_y}\circ\tau} e^{-A_t}h(X_t)d\tilde A_t \bigg\vert\f_{\eta_y} \right]\right] \notag\\
    &+ \Ex_x \left[ \ind{\eta_y > \tau }  \int_0^\tau e^{-A_t}h(X_t)d\tilde A_t  \right]\notag\\
    =& \Ex_x\left[\ind{\eta_y \le \tau } \int_0^{\eta_y} e^{-A_t}h(X_t)d\tilde A_t \right] +\Ex_x\left[\ind{\eta_y \le \tau } \Ex_{X_{\eta_y}}\left[\int_0^{\eta_y} e^{-A_t}h(X_t)d\tilde A_t \right] \right]\notag\\
    &+ \Ex_x \left[ \ind{\eta_y > \tau }  \int_0^\tau e^{-A_t}h(X_t)d\tilde A_t  \right]\notag\\
    =& \Ex_x[\ind{\eta_y \le \tau }]F(y)+\Ex_x\left[\ind{\eta_y \le \tau } \int_0^{\eta_y} e^{-A_t}h(X_t)d\tilde A_t \right]  \notag\\
    &+ \Ex_x \left[ \ind{\eta_y > \tau }  \int_0^\tau e^{-A_t}h(X_t)d\tilde A_t  \right]. \label{eq:markov1}
\end{align} Applying \eqref{eq:itomeyer} to \eqref{eq:markov1} we obtain
\begin{align}
    &\vert F(x) - F(y) \vert \notag \\
    =& \bigg \vert F(y) (1 - \Ex_x[\ind{\eta_y \le \tau }]) + \Ex_x\left[\ind{\eta_y \le \tau } \int_0^{\eta_y} e^{-A_t}h(X_t)d\tilde A_t \right] + \Ex_x \left[ \ind{\eta_y > \tau }  \int_0^\tau e^{-A_t}h(X_t)d\tilde A_t  \right] \bigg \vert \notag\\
    \le& \vert F(y)\vert (1 - \Ex_x[\ind{\eta_y \le \tau }])  +  \Big(\sup_{z\in I} \vert h(z) \vert \Big) \Ex_x[\ind{\eta_y\le \tau} \tilde A_{\eta_y-}] +  \Big(\sup_{z\in I} \vert h(z) \vert \Big) \Ex_x[\ind{\eta_y>\tau}\tilde A_{\tau-}]\notag \\
    \le& \vert F(y) \vert (1- \PR_x(\eta_y\le \tau))  +  \Big(\sup_{z\in I} \vert h(z) \vert \Big) \Ex_x[\tilde A_{(\eta_y \wedge \tau)-}] +  \Big(\sup_{z\in I} \vert h(z) \vert \Big) \Ex_x[\tilde A_{(\eta_y \wedge \tau)-}] \notag\\
    =& \vert F(y) \vert (1- \PR_x(\eta_y\le \tau))  +  2\Big(\sup_{z\in I} \vert h(z) \vert \Big) (2\Ex_x[ \Psi(X_{\eta_y \wedge \tau})] -2 \Psi(x))\notag \\
    &-2\Big(\sup_{z\in I} \vert h(z) \vert \Big)\Ex_x\bigg[ \int_0^{\eta_y \wedge \tau} 2\psi(X_s) \mu(X_s) ds - \int_0^{\eta_y \wedge \tau} 2 \psi(X_s) \sigma(X_s) dW_s \bigg]\notag\\
    \le &\vert F(y) \vert (1- \PR_x(\eta_y\le \tau))  +  4\Big(\sup_{z\in I} \vert h(z) \vert \Big) \vert \Ex_x[ \Psi(X_{\eta_y \wedge \tau})] - \Psi(x) \vert\notag \\
    &+4\Big(\sup_{z\in I} \vert h(z) \vert \Big)  \Ex_x[\eta_y\wedge \tau] \Big( \sup_{z\in J} \vert \psi(z) \mu(z)\vert \Big). \label{eq:est2}
\end{align} for all $x\in J$ and all $y\in I$. For the last step note that $Y=(Y_t)_{t\in[0,\infty)}$, $Y_t:=\int_0^t \ind{s< \eta_y\wedge \tau} \psi(X_s)\sigma(X_s)dW_s$, $t\ge0$ is a true martingale. Thus, invoking $\Ex_x[\tau]<\infty$ the expectation vanishes by the optional sampling theorem.

By symmetry it suffices to show that all the summands on the right hand side of \eqref{eq:est2} go to 0, whenever $J \ni x\to y\in  \overline{J}$. The first summand has already been dealt with. 

We denote the scale function of $X$ by $s$. We have
\begin{align*}
    \Ex_x[\Psi(X_{\eta_y \wedge \tau})] 
    &=\begin{cases}
        \Psi(y) \PR_x(X_{\eta_y \wedge \tau} =y) + \Psi(a) \PR_x(X_{\eta_y \wedge \tau} =a), &\text{if }\, a\le x\le y,\\
        \Psi(y) \PR_x(X_{\eta_y \wedge \tau} =y) + \Psi(b) \PR_x(X_{\eta_y \wedge \tau} =b), &\text{if }\, y\le x\le b
    \end{cases}\\
    &= \begin{cases}
        \Psi(y) \frac{s(x)-s(a)}{s(y)-s(a)}  + \Psi(a) \frac{s(y)-s(x)}{s(y)-s(a)}, &\text{if }\, a\le x\le y\\
        \Psi(y) \frac{s(b)-s(x)}{s(b)-s(y)}  + \Psi(b) \frac{s(x)-s(z)}{s(b)-s(y)}, &\text{if }\, y\le x\le b
    \end{cases} 
\end{align*} for all $x\in\overline{J}$. $s$ is continuous, cf. \cite[Section 4.2]{ito1974diffusion}. Thus, the function $\overline{J} \ni x\mapsto \Ex_x[\Psi(X_{\eta_y \wedge \tau})] $ is continuous with $\lim_{J \ni x \to y}\Ex_x[\Psi(X_{\eta_y \wedge \tau})] =\Psi(y)$, i.e. $\lim_{J \ni x \to y}\Ex_x[\Psi(X_{\eta_y \wedge \tau})] -\Psi(x)=0$ by continuity of $\Psi$. This means the second summand in \eqref{eq:est2} vanishes for $J \ni x \to y$.

For $x\in \overline{J}$ and $y\in J$ we either have $\Ex_x[\eta_y \wedge \tau ]= \Ex_x[\eta_y\wedge \eta_a]$ or $\Ex_x[\eta_y \wedge \tau ]= \Ex_x[\eta_y\wedge \eta_b]$ depending on whether $y\le x\le b$ or $a\le x \le y$. The mappings $[y,b]\ni x \mapsto \Ex_x[\eta_y\wedge \eta_a]$ and $[a,y]\ni x \mapsto \Ex_x[\eta_y\wedge \eta_a]$ are continuous with $\Ex_y[\eta_y\wedge \eta_a]=0$ and $\Ex_y[\eta_y\wedge \eta_a]$ by \cite[Section 4.2, 20a), 20b)]{ito1974diffusion} since $X$ is a regular diffusion. Thus the second and third summand vanish for $y\to x$ which proves the claim.
\end{proof}

\begin{lemma}\label{lemma:thm9.7}
Let $Y=(Y_t)_{t\in [0,\infty)}$ be a standard process with continuous paths and values in an interval $J\subset \R$. Let $A$ be the weak infinitesimal operator of $Y$ and $\psi:J \to [0,\infty)$ a bounded measurable function with the property $\Ex_{\bigcdot}[\psi(Y_t)]\to \psi$ weakly for $t\searrow 0$. We set $\varphi:=(\varphi_t)_{t\in[0,\infty)}$,
    \begin{align*}
        \varphi_t:= \int_0^t \psi(Y_s)ds, \quad t\ge0.
    \end{align*}
    Let $\tilde A$ denote the weak infinitesimal operator of the transition function $\tilde \PR$ given by
    \begin{align*}
        \tilde \PR(t,x,\Gamma) := \Ex_x[\inda{\Gamma}(Y_t) e^{-\varphi_t}], \quad t\ge 0, x\in J, \Gamma \in \mathcal{B}(J).
    \end{align*} Now 
    \begin{align*}
        \mathcal{D}_{\tilde A}\cap \ccc^0(J)\cap \B(J) = \mathcal{D}_{A}\cap \ccc^0(J) \cap \B(J)
    \end{align*} and 
    \begin{align*}
        \tilde A f = A f - \psi f
    \end{align*} for all $f\in  \mathcal{D}_{\tilde A}\cap \ccc^0(J)\cap \B(J)$.
\end{lemma}
\begin{remark*}
    This lemma is an adaptation of \cite[Theorem 9.7]{dynkin_markov} in the weak formulation outlined in subsequent remark, cf.\ \cite[p.\ 299]{dynkin_markov}. The original weak formulation of the theorem yields $\mathcal{D}_{\Tilde{A}}=\mathcal{D}_A$ provided that the additional condition (9.62) from \cite[p.\ 299]{dynkin_markov} is satisfied.   
\end{remark*}
\begin{proof}
    Clearly $\sup_{x\in J} \Ex_x[\varphi_t] \to 0$ for $t \searrow 0$. Thus the weak formulation of \cite[Theorem 9.6]{dynkin_markov} from the subsequent Remark 1, \cite[p.\ 297]{dynkin_markov} yields
    \begin{align}
        cE - \tilde A = (c\id - A) (\id+ S_c) \label{eq:thm9.6}
    \end{align} for some $c >0$, where $\id$ denotes the identity operator and $S_c$ maps a measurable function $f:J\to \R$ to the function $S_c f(x):= \Ex_x[\int_0^\infty e^{-ct}f(Y_t) \psi(Y_t) dt]$, $x\in J$, whenever the right hand side is well defined and finite.  We denote the resolvent operator of the semigroup generated by (the transition function of) $Y$ by $R_c$, $c>0$, cf.\ \cite[Subsection 5.1]{dynkin_markov}. Now by definition 
    \begin{align*}
        S_cf(x)=R_c(f\psi)(x), \quad x \in J.
    \end{align*} Together with \eqref{eq:thm9.6} this implies 
    \begin{align}
        \mathcal{D}_{\tilde A}= \{ f \in \B(J): f + R_c(f\psi)\in \mathcal{D}_A\}.\label{eq:9.61}
    \end{align}
    
    Let $f\ccc^0(J)\cap \B(J)$. By \cite[Theorem 1.7]{dynkin_markov} $R_c(f\psi)\in \mathcal{D}_A$ if $\Ex_{\bigcdot}[f(Y_t)\psi(Y_t)]\to f\psi$ weakly for $t\searrow 0$. However, for each $x\in J$
    \begin{align*}
        &\vert \Ex_{x}[f(Y_t)\psi(Y_t)]- f(x)\psi(x) \vert 
        \le \vert \Ex_{x}[(f(Y_t)- f(x))\psi(Y_t)]  \vert+ \vert f(x)(\Ex_{x}[\psi(Y_t)] - \psi(x))\vert\\
        \le& \Big(\sup_{y\in I} \psi(y) \Big) \Ex_{x}[(f(Y_t)- f(x))] + \vert f(x)(\Ex_{x}[\psi(Y_t)] - \psi(x))\vert \stackrel{t\searrow 0}{\longrightarrow}0 
    \end{align*} by dominated convergence using $f\in \ccc^0(J)\cap \B(J)$ and due to the assumption on $\psi$. As $f$ and $\psi$ are both bounded, dominated convergence also yields $\Ex_{\bigcdot}[f(Y_t)\psi(Y_t)]\to f\psi$ weakly for $t\searrow 0$ as desired. 
    
    If $f\in \mathcal{D}_{\tilde A}\cap \ccc^0(J)\cap \B(J)$, then $f + R_c(f\psi)\in \mathcal{D}_A$ by \eqref{eq:9.61}. $R_c(f\psi)\in \mathcal{D}_A$ has already been shown and so by linearity $f= (f + R_c(f\psi)) - R_c(f\psi)\in \mathcal{D}_A$. 

    If $f\in \mathcal{D}_{A}\cap \ccc^0(J)\cap \B(J)$, then since $R_c(f\psi)\in \mathcal{D}_A$ also $f+R_c(f\psi)\in \mathcal{D}_A$. By \eqref{eq:9.61} this implies $f\in \mathcal{D}_{\tilde A}$.
    
    To prove the final claim let $f\in \mathcal{D}_{\tilde A}\cap \ccc^0(J)\cap \B(J)$. Once again by \cite[Theorem 9.6]{dynkin_markov}
    \begin{align*}
        (c\id - \tilde A) f = (c\id - A) (\id+ S_c)f = (c \id - A)f + (c\id - A) R_c(\psi f) = (c\id - A) f + \psi f,
    \end{align*} i.e.\ $\tilde A f = A f - \psi f$.
\end{proof}

\begin{lemma}\label{lemma:thm13.11}
Let $\mathfrak A$ denote the characteristic operator of $X$. Moreover, let $J\subset I$ be an open interval with $\tau^J<\infty$ $\PR_x$-a.s.\ for all $x\in J$ and $\psi:J \to [0,\infty)$ a bounded measurable function. Assume that $\Ex_{\bigcdot}[\psi(X_{t\wedge \tau^J})]\vert_J\to \psi$ weakly for $t\searrow 0$. We extend $\psi$ to $I$ by 0 outside of $J$, set $\tau:=\tau^J$ and
    \begin{align*}
        f(x):=\Ex_x\left[e^{-\int_0^{\tau}\psi(X_s)ds} g(X_{\tau}) \right]
    \end{align*} for all $x\in I$. Then $f$ is continuous with
    \begin{align*}
        \mathfrak A f(x) - \psi(x)  f (x) = 0
    \end{align*} and in particular $f\in \mathcal{D}_{\mathfrak A}(y)$ for all $y\in J$.
\end{lemma}
\begin{remark*}
    The proof of this result is analogous to the proof of \cite[Theorem 13.11]{dynkin_markov2} up to the part where we apply Lemma \ref{lemma:thm9.7} instead of its counterpart \cite[Theorem 9.7]{dynkin_markov}. 
\end{remark*}
\begin{proof} Note that $f$ is continuous by Lemma \ref{lemma:cont} \eqref{lemma:cont1}. 
    
    Let $\hat X = (\hat X_t)_{t\in[0,\infty)}$, be given by $\hat X_t(\omega) := X_{t\wedge\tau(\omega)}(\omega)$, $t\in [0,\infty)$. By \cite[Theorem 10.3]{dynkin_markov} $\hat X=(\hat X_t, \infty, \f_t, \PR_x)$ is a standard process. 
    Let $\alpha = (\alpha_t)_{t \in [0 ,\infty)}$ be given by
    \begin{align*}
        \alpha_t := e^{-\int_0^t \psi(\hat X) ds}, \quad t \ge 0.
    \end{align*} Clearly $\alpha$ is a continuous homogeneous multiplicative functional of $\hat X$. We let $\tilde X = (\tilde X_t, \tilde \zeta, \tilde \f_t, \tilde \PR_x)$ denote the standardized $\alpha$-subprocess of $\hat X$, cf. \cite[Subsection 10.10, Subsection 10.12, Subsection 10.17]{dynkin_markov}. In particular this means that $\tilde X$ is a process on $(\Omega \times [0,\infty], \tilde \f \otimes \mathcal{B}([0,\infty]))$ with $\tilde \f$ denoting the $\sigma$-algebra generated by $\tilde \f_t, t\ge0$, $\tilde \zeta (\omega,t):=  t$, $\tilde X_t :=  \hat X_t(\omega), t < \zeta(\omega,t)$ and $\tilde \PR_x$ are such that 
    \begin{align}
        \tilde \PR_x(\tilde X_t \in \Gamma)= \Ex_x[\inda{\Gamma}(\hat X_t)\alpha_t] \label{eq:10.18}
    \end{align} for all $\Gamma\in\mathcal{B}(I)$. Moreover, $\tilde X$ is a standard process, cf. \cite[Theorem 10.7]{dynkin_markov}. 
    Now let $F:I\to \R$ be bounded and measurable and $x\in J$. We denote the expectation corresponding to $\tilde \PR_x$ by $\tilde \Ex_x$. Now \eqref{eq:10.18} extends to 
    \begin{align*}
        \tilde \Ex_x[F(\tilde X_t)] = \Ex_x[F(\hat X_t) \alpha_t].
    \end{align*} Using $\hat X_t= X_{t\wedge \tau}$ and $\alpha_t= \alpha_{t\wedge \tau}$ for all $t\ge0$ since $\psi=0$ outside of $J$, we find that
    \begin{align*}
        \tilde \Ex_x[F(\tilde X_t)] &= \Ex_x[\ind{t<\tau}F(\hat X_t) \alpha_t]+ \Ex_x[\ind{t\ge \tau}F(X_\tau) \alpha_\tau].
    \end{align*} Since we assumed $\tau< \infty$ $\PR_x$-a.s.\ and that $F$ is bounded, dominated convergence yields
    \begin{align}
        \lim_{t\to \infty} \tilde \Ex_x[F(\tilde X_t)] = \Ex_x[F(X_\tau) \alpha_\tau]. \label{eq:13.53}
    \end{align}
    From now on let $F$ be bounded and measurable such that $F=g$ on the boundary $\partial J\subset I$ of $J$. By definition of $f$, followed by dominated convergence applied to \eqref{eq:13.53} we infer
    \begin{align*}
        f(x)= \Ex_x[g(X_\tau) \alpha_\tau ] = \lim_{t\to \infty} \tilde \Ex_x[F(\tilde X_t)].
    \end{align*} Thus, for the weak infinitesimal generator $\tilde A$ of $\tilde X$ we have $\tilde Af=0$ and in particular $f\in \mathcal{D}_{\tilde A}$. 
    
    We denote the weak infinitesimal operator of $\hat X$ by $\hat A$. Note that $\hat X$ is a standard process with continuous paths and since $\Ex_{y}[\psi(X_{t\wedge\tau})] = \psi(y)$ for all $y\in I\setminus (a,b)$ we still have $\Ex_{\bigcdot}[\psi(X_{t\wedge\tau})] \to \psi$ weakly for $t\searrow 0$ despite extending $\psi$. Due to \eqref{eq:10.18}, Lemma \ref{lemma:thm9.7} can be applied to the process $\hat X$, $\tilde \PR_x$ and the extension of $\psi$ to provide $\tilde A G=\hat A G - \psi G$ for all $G \in \mathcal{D}_{\tilde A}\cap \ccc^0(I)\cap \B(I)=\mathcal{D}_{\hat A}\cap \ccc^0(I)\cap \B(I)$. In particular since $f$ is continuous,
    \begin{align}
        \hat A f - \psi f= \tilde A f =0\label{eq:13.54}
    \end{align} and $f\in \mathcal{D}_{\hat A}$. If we denote the characteristic operator of $\hat X$ by $\hat{\mathfrak A}$, \cite[Lemma 5.6]{dynkin_markov} provides $f\in \hat{\mathfrak A}(x)$ and $\hat{\mathfrak A} f(x)= \psi(x) f(x)$ due to \eqref{eq:13.54}. As $J$ is open, by definition of the characteristic operator $\mathcal{D}_{\hat{\mathfrak A}}(x) = \mathcal{D}_{\mathfrak A}(x)$ and $\hat{\mathfrak A} f(x) = \mathfrak A f(x)$. In particular $f\in \mathcal{D}_{\mathfrak A}$ and $\mathfrak A f(x)-\psi(x)f(x)=0$, which proves the claim.
\end{proof}

\begin{lemma}\label{lemma:thm13.12}
    Let $\mathfrak A$ denote the characteristic operator of $X$. Moreover, let $J\subset I$ be an open interval with $\Ex_x[\tau^J]<\infty$ for all $x\in J$, $\psi:J \to [0,\infty)$ bounded and measurable and $h:J\to \R$ bounded and measurable. Assume that $\Ex_{\bigcdot}[\psi(X_{t\wedge \tau^J})]\vert_J\to \psi$ weakly for $t\searrow 0$ and $\Ex_{\bigcdot}[h(X_{t\wedge \tau^J})e^{-\int_0^{t\wedge\tau} \psi(X_s)ds }]\vert_J\to h$ weakly for $t\searrow 0$ respectively. We extend the functions $\psi, h$ to $I$ by 0 outside of $J$, set $\tau:=\tau^J$ and
    \begin{align*}
        F(x):= \Ex_x\left[\int_0^\tau e^{-\int_0^s \psi(X_s) ds} h(X_s) dt \right]
    \end{align*} for all $x\in I$. Then $F$ is continuous with
    \begin{align*}
        \mathfrak A F(x) - \psi(x) F(x) = - h (x)
    \end{align*} and in particular $F\in\mathcal{D}_{\mathfrak A}(y)$ for all $x\in J$.
\end{lemma}
\begin{remark*}
    Once again, the proof of this lemma is mostly analogous to its counterpart \cite[Theorem 13.12]{dynkin_markov2}. The first difference is that here $\tilde \Ex_{\bigcdot}[h(\tilde X_t)] \to h$ weakly for $t\searrow 0$ is assumed directly and not inferred from continuity. The other is that, just as for Lemma \ref{lemma:thm13.11}, we rely on Lemma \ref{lemma:thm9.7} instead of its counterpart \cite[Theorem 9.7]{dynkin_markov2}. 
\end{remark*}

\begin{proof}
    Note that $F$ is continuous by Lemma \ref{lemma:cont} \eqref{lemma:cont2}.

    Let $\hat X=(\hat X_t, \infty, \f_t, \PR_x)$, $\alpha=(\alpha_t)_{t\in[0,\infty)}$ and $\tilde X=(\tilde X_t,\tilde \zeta, \tilde \f_t, \tilde \PR_x)$ be the processes constructed in the proof of Lemma \ref{lemma:thm13.11}, $\hat A$, $\tilde A$ their weak infinitesimal generators and $\hat{\mathfrak A}$, $\tilde{\mathfrak A}$ their characteristic operators.
    
    Let $x\in I$. Once again, we denote the expectation corresponding to $\tilde \PR_x$ by $\tilde \Ex_x$. Using the definitions of $\tilde X_t$ and $\hat X_t$, \cite[Theorem 10.5, (10.20)]{dynkin_markov} and the assumption $h=0$ outside of $J$ we obtain
    \begin{align}
        \tilde \Ex_x[h(\tilde X_t)]=\tilde \Ex_x[h( \hat X_t ) \ind{\tilde \zeta > t}] = \Ex_x[h( \hat X_t ) \alpha_t] =  \Ex_x[\ind{\tau>t} h( X_t ) \alpha_t].\label{eq:13.56}
    \end{align} for all $t\ge0$. Let $\tilde R$ denote the potential of $\tilde X$. Using the definition of the potential in the first step, \eqref{eq:13.56} in the second, Fubini's theorem in the third and the definition of $F$ is the last step we obtain
    \begin{align*}
        \tilde R h(x)&= \int_0^\infty \tilde \Ex_x[h(\tilde X_t)] dt = \int_0^\infty \Ex_x[\ind{\tau>t} h( X_t ) \alpha_t] dt = \Ex_x\left[ \int_0^\infty  \ind{\tau>t}h(X_t)\alpha_t dt\right] \\
        &=\Ex_x\left[ \int_0^\tau  h(X_t)\alpha_t dt\right]=F(x).
    \end{align*} Thus, since $x\in I$ was arbitrary, $\tilde R h = F$.

    We have $ \tilde \Ex_y[h(\tilde X_t)]= \Ex_y[h(\hat X_t)\alpha_t]=0$ for all $t\ge 0$ and all $y\in I\setminus J$ since $h=0$ outside $J$ and $\hat X_t=y $ $\PR_y$-a.s. Combined with $ \tilde \Ex_{\bigcdot}[h(\tilde X_t)]  =\Ex_{\bigcdot}[h(\hat X_t)\alpha_t]=\Ex_{\bigcdot}[h(X_{t\wedge \tau^J})e^{-\int_0^{t\wedge\tau} \psi(X_s)ds }]$ and the assumption on $h$, this implies $\tilde \Ex_{\bigcdot}[h(\tilde X_t)]\to h$ weakly for $t\searrow 0$. Now \cite[Theorem 1.7$'$]{dynkin_markov} provides $F\in \mathcal{D}_{\tilde A}$ and $\tilde A F= -h$.
    
    As mentioned in the proof of Lemma \ref{lemma:thm13.11}, $\tilde X$ is a standard process and due to \eqref{eq:10.18}, the preconditions of Lemma \ref{lemma:thm9.7} are met for $\hat X$, $\tilde \PR_x$ and $\psi$. The lemma yields $\tilde A G=\hat A G - \psi G $ for all $G\in \mathcal{D}_{\hat A}\cap \ccc^0(I)\cap \B(I)=\mathcal{D}_{\tilde A}\cap \ccc^0(I)\cap \B(I)$. In particular since $F$ is continuous
    \begin{align}
        \hat A F- \psi F = \tilde A F= -h \label{eq:13.54_2}
    \end{align} and $F\in \mathcal{D}_{\hat A}$. 
    
    \cite[Lemma 5.6]{dynkin_markov} provides $F\in \hat{\mathfrak A}(x)$ and $\hat{\mathfrak A} F(x)= \psi(x) F(x) -h(x)$ due to \eqref{eq:13.54_2}. As $J$ is open, by definition of the characteristic operator $\mathcal{D}_{\hat{\mathfrak A}}(x) = \mathcal{D}_{\mathfrak A}(x)$ and $\hat{\mathfrak A} F(x) = \mathfrak A F(x)$. In particular $F\in \mathcal{D}_{\mathfrak A}$ and $\mathfrak A F(x)-\psi(x)f(x)-h(x)=0$ which proves the claim.
\end{proof}

\begin{proof}(of Theorem \ref{thm:main})

\eqref{cont} We first calculate the conditional distribution of $\tau$ under $\PR_x$ given $\f^X_\infty$. Since $A_t$ is $\f_\infty^X$-measurable and $E$ is independent from $X$ we have
    \begin{align*}
        \PR_x(\tau_2>t \vert\f^X_\infty) = \PR_x(A_t<E\vert \f^X_\infty) = e^{-A_t}.
    \end{align*} for all $t\ge 0$. Thus $d(-e^{-A_{\bigcdot}})$ (i.e.\ the Markov kernel $\Omega\times \mathcal{B}([0,\infty])\ni (\omega, (a,b]) \mapsto e^{-A_a(\omega)}-e^{-A_b(\omega)} $) is a regular version of the conditional distribution of $\tau$ given $\f^X_\infty$. Applying a change of variables for finite variation processes, cf. \cite[p. 42]{protter2005stochastic}, and using that $A$ only jumps in $\tau^D$, we find
    \begin{align*}
        e^{-A_a(\omega)}-e^{-A_b(\omega)} = \int_{(a,b\wedge \tau^D)} e^{-A_{s-}}dA_{\bigcdot}(s)  + e^{-A_{\tau^D-}} - e^{-A_{\tau^D}}.
    \end{align*}
With the distribution of $\tau$, the fact $\tau^{(a,b)} \le \tau^D$ and the Markov property we obtain
\begin{align}
    &J_{\tau}(x) \notag\\
    =& \Ex_x[\ind{ \tau^{(a,b)} >\tau } e^{-r \tau} g(X_\tau)]+ \Ex_x[\ind{ \tau^{(a,b)} \le \tau } e^{-r (\theta_{\tau^{(a,b)}}\circ \tau + \tau^{(a,b)}) } g(X_{\theta_{\tau^{(a,b)}}\circ \tau + \tau^{(a,b)}}) ]  \notag\\
    =&\Ex_x[\Ex_x[ \ind{ \tau^{(a,b)} >\tau } e^{-r \tau} g(X_\tau) \vert \f_\infty^X] ] + \Ex_x[ \ind{ \tau^{(a,b)} \le \tau }  e^{-r \tau^{(a,b)} }\Ex_x [ e^{-r \theta_{\tau^{(a,b)}}\circ \tau  } g(X_{\theta_{\tau^{(a,b)}}\circ \tau + \tau^{(a,b)}})  \vert \f_{\tau_{(a,b)}}] ]\notag \\
    =&\Ex_x\left[  \int_{[0,\tau^{(a,b)})}  e^{-r t} g(X_t) d(-e^{-A_t})\right] + \Ex_x[ \ind{ \tau^{(a,b)} \le \tau }  e^{-r \tau^{(a,b)} } \Ex_{X_{\tau^{(a,b)}}} [e^{-r  \tau  } g(X_{\tau}) ] ] \notag \\
    =& \Ex_x\left[  \int_{[0,\tau^{(a,b)})}  e^{-r t} g(X_t) e^{-A_{t}} dA_t \right] + \Ex_x[ \ind{ \tau^{(a,b)} \le \tau }  e^{-r \tau^{(a,b)} } J_{\tau}(X_{\tau^{(a,b)}})] \notag\\
    =& \Ex_x\left[  \int_{[0,\tau^{(a,b)})}   e^{-r t-A_{t}} g(X_t) dA_t \right] + \Ex_x[e^{-r \tau^{(a,b)} } J_{\tau}(X_{\tau^{(a,b)}})\Ex_x[ \ind{ \tau^{(a,b)} \le \tau }   \vert \f_\infty^X]] \notag\\
    =& \Ex_x\left[  \int_{[0,\tau^{(a,b)})}  e^{-r t-A_{t}} g(X_t) dA_t \right] + \Ex_x[e^{-r \tau^{(a,b)} } J_{\tau}(X_{\tau^{(a,b)}})e^{-A_{\tau^{(a,b)}-}}] \notag\\
     =& \Ex_x\left[  \int_{[0,\tau^{(a,b)})}  e^{-rt- A_{t}} g(X_t) dA_t \right] + \Ex_x[e^{-r \tau^{(a,b)}- A_{\tau^{(a,b)}-}} (\ind{a,b}J_{\tau})(X_{\tau^{(a,b)}})]  \label{eq:zerlegung}
\end{align} for all $x\in [a,b]$. With that the claim follows from Lemma \ref{lemma:cont}.

\eqref{diff2} By the occupation density formula \cite[p.\ 104]{RogersWilliamsVol2}, $\lambda\vert_{(a,b)}= \frac{\psi}{\sigma}dx$ implies $A_t= \int_0^t \psi(X_s)ds$ for all $t\in [0, \tau^{(a,b)})$, $\PR_x$-a.s. for all $x\in I$. Setting $\tilde \psi :=\psi +r$ by \eqref{eq:zerlegung} we infer
\begin{align}
    J_{\tau}(x) &= \Ex_x\left[  \int_{[0,\tau^{(a,b)})}  e^{-rt- \int_0^t\psi(X_s)ds} \psi(X_t)g(X_t) dt \right] + \Ex_x\bigg[e^{-r \tau^{(a,b)}-\int_0^{\tau^{(a,b)}}\psi(X_s)ds} (\ind{a,b}J_{\tau})(X_{\tau^{(a,b)}})\bigg]\notag \\
    &=\Ex_x\left[  \int_{[0,\tau^{(a,b)})}  e^{- \int_0^t \tilde \psi(X_s)ds} \psi(X_t)g(X_t) dt \right] + \Ex_x\bigg[e^{-\int_0^{\tau^{(a,b)}}\tilde \psi(X_s)ds} (\ind{a,b}J_{\tau})(X_{\tau^{(a,b)}})\bigg]\label{eq:zerl2}
\end{align} for all $x\in (a,b)$.
Whenever $a,b\in I$, the functions $\tilde \psi$ and $\psi \cdot g\vert_{(a,b)}$ are Hölder continuous on $(a,b)$ and the dynamics of $X$ on $[a,b]$ coincide with the dynamics of a canonical diffusion. If $a\not \in I$ or $b\not \in I$, we cover $(a,b)$ by subintervals that are bounded away from the boundary of $I$ and argue separately for each subinterval. Now the claim follows from \cite[Theorem 13.16]{dynkin_markov2}.

\eqref{diff} We denote the characteristic operator of $X$ by $\mathfrak A$. By \eqref{cont} $J_\tau\vert_{(a,b)}\in \ccc^0([a,b])$. By \eqref{diff2} $J_\tau\vert_{(x_{i-1},x_i)}\in \ccc^2((x_{i-1},x_i))$ for all $i\in \{1,...,n\}$. Set $x:=x_i$ for some $i\in \{1,...,n-1\}$ and $\tau_h:=\tau^{(x-h,x+h)}$. By a generalized Itô formula, cf.\ \cite{Peskir2007}, we obtain
    \begin{align}
        &\frac{\Ex_x[J_{\tau}(X_{\tau_h})]- J_{\tau}(x)  }{\Ex_x[\tau_h]}\notag\\
        =& \frac{\Ex_x\big[  \int_0^{\tau_h}\ind{X_s\neq x} \mathcal AJ_{\tau}(X_s)ds + \frac{1}{2}\int_0^{\tau_h} \partial_x J_{\tau}(x+)- \partial_x J_{\tau}(x-)dl^x_s\big]}{\Ex_x[\tau_h]}\notag\\
        =& \frac{\Ex_x\big[  \int_0^{\tau_h} \ind{X_s\neq x}\mathcal AJ_{\tau}(X_s)ds\big]}{\Ex_x[\tau_h]}+  \frac{1}{2}(\partial_x J_{\tau}(x+)- \partial_x J_{\tau}(x-))\frac{\Ex_x[l^x_{\tau_h}]}{\Ex_x[\tau_h]} \label{eq:finitelim}
    \end{align} for all $h \in (0, (x-x_{i-1})\wedge (x_{i+1}-x))$ with $\partial_x J_{\tau}(x+)$ and $\partial_x J_{\tau}(x-)$ denoting the right- and left derivative of $J_\tau$ at $x$. By \eqref{eq:zerl2}, Lemma \ref{lemma:thm13.11} and Lemma \ref{lemma:thm13.12}
    \begin{align*}
        &\mathfrak A J_\tau (y)\\
        =& \mathfrak A \Ex_y\left[  \int_{[0,\tau^{(x_{i-1},x)})}  e^{- \int_0^t \tilde \psi(X_s)ds} \psi(X_t)g(X_t) dt \right]+ \mathfrak A \Ex_y\bigg[e^{-\int_0^{\tau^{(x_{i-1},x)}}\tilde \psi(X_s)ds} (\ind{x_{i-1},x}J_{\tau})(X_{\tau^{(x_{i-1},x)}})\bigg] \\
        =& \psi(y) \Ex_y\left[  \int_{[0,\tau^{(x_{i-1},x)})}  e^{- \int_0^t \tilde \psi(X_s)ds} \psi(X_t)g(X_t) dt \right] - g(y) \\
        &+ \psi(y) \Ex_y\bigg[e^{-\int_0^{\tau^{(x_{i-1},x)}}\tilde \psi(X_s)ds} (\ind{x_{i-1},x}J_{\tau})(X_{\tau^{(x_{i-1},x)}})\bigg] \\
        =&\psi(y) J_\tau(y) -g(y)
    \end{align*} for all $y\in (x_{i-1},x)$. Due to the Hölder conditions on $g,\psi$ and continuity of $J_\tau$ we find that $\lim_{y\nearrow x}\mathfrak A J_\tau(y)$ is finite. Analogously we obtain that $\lim_{y\searrow x}\mathfrak A J_\tau(y)$ is finite. Thus by \cite[Lemma A.4]{christensen2023time} the first summand of the right hand side of \eqref{eq:finitelim} has a finite limit for $h\searrow 0$. By \cite[Lemma 26]{BodnariuChristensenLindensjoe2022} we have $\lim_{h\searrow 0}\frac{\Ex_x[L^x_{\tau_h}]}{\Ex_x[\tau_h]}=\infty$. Thus, the right hand side has a finite limit for $h\searrow 0$ if and only if $\partial_x J_{\tau}(x+)- \partial_x J_{\tau}(x-)=0$. We finish the proof by showing that $\lim_{h\searrow 0}\frac{\Ex_x[J_{\tau}(X_{\tau_h})]- J_{\tau}(x)}{\Ex_x[\tau_h]}$ exists in $\R$. For that set 
    \begin{align*}
        \hat \psi (y)= \frac{\tilde \psi(y+) + \tilde \psi(y-) }{2} 
    \end{align*} for all $y\in (a,b)$. By \cite[Lemma 5.5]{bayraktar2022equilibria} and dominated convergence
    \begin{align*}
        &\vert \Ex_y[\ind{\tau^{(a,b)}>t}\hat{\psi}(X_t)] -\hat \psi (y) \vert
        =\left\vert \Ex_y[\ind{\tau^{(a,b)}>t}\hat{\psi}(X_t)] -\frac{\psi(y+)+\psi(y-)}{2} \right\vert\\
        \le&  \vert \Ex_y[\ind{\tau^{(a,b)}>t}\hat{\psi}(X_t)] - \Ex_y[\hat{\psi}(X_t)] \vert +\vert \Ex_x[\ind{X_t>y}(\hat{\psi}(X_t)-\psi(y+))]\vert\\
        &+ \vert \Ex_y[\ind{X_t<y}(\hat{\psi}(X_t)-\psi(x-))]\vert   + \left\vert \psi(y+)\left(\Ex_y[\ind{X_t>y}]-\frac{1}{2}\right)\right\vert \\
        &+\left\vert \psi(y-)\left(\Ex_y[\ind{X_t<y}]-\frac{1}{2} \right)\right\vert \stackrel{t\searrow 0}{\longrightarrow} 0.
    \end{align*} Thus also $\int_{(a,b)}\Ex_y[\ind{\tau^{(a,b)}>t}\hat{\psi}(X_t)] d\mu(y) \to \int_{(a,b)}\hat{\psi}(y)d\mu(y) $ for all finite measures $\mu$ on $(a,b)$, i.e.\ $\Ex_{\bigcdot}[\ind{\tau^{(a,b)}>t}\hat{\psi}(X_t)]\to \hat \psi$ weakly for $t\searrow 0$. Using $\Ex_y[\int_0^{t}\ind{X_s\in \{x_1,...,x_{n-1}\}}ds]=0$ for all $t\ge 0$, \eqref{eq:zerl2} implies
    \begin{align*}
        J_\tau(y)=\Ex_y\left[  \int_{[0,\tau^{(a,b)})}  e^{- \int_0^t \hat \psi(X_s)ds} g(X_t) dt \right] + \Ex_y\bigg[e^{-\int_0^{\tau^{(a,b)}}\hat \psi(X_s)ds} (\ind{a,b}J_{\tau})(X_{\tau^{(a,b)}})\bigg]
    \end{align*} for all $y \in(a,b)$. Thus by Lemma \ref{lemma:thm13.11} and Lemma \ref{lemma:thm13.12} we have $J_\tau\in \mathcal{D}_{\mathfrak A}(y)$ for all $y\in (a,b)$. By definition this means that $\lim_{h\searrow 0}\frac{\Ex_x[J_{\tau}(X_{\tau_h})]- J_{\tau}(x)}{\Ex_x[\tau_h]}$ exists, which finishes the proof.    
\end{proof}

\bibliographystyle{abbrv}
\bibliography{literature}

\end{document}